\documentclass{amsart}
\usepackage{amsmath, amsthm, amscd, amsfonts, amssymb, graphicx, color}
\usepackage[bookmarksnumbered, colorlinks, plainpages]{hyperref}
\newtheorem{thm}{Theorem}[section]
\newtheorem{cor}[thm]{Corollary}

\newtheorem{prop}[thm]{Proposition}

\newtheorem{exa}[thm]{Example}

\newcommand{\no}{\nonumber}

\newcommand{\ri}{\rightarrow}

\newcommand{\al}{\alpha}

\newcommand{\la}{\lambda}

\newcommand{\ph}{\varphi}

\newcommand{\ep}{\epsilon}
\newcommand{\be}{\begin{eqnarray}}
\newcommand{\ee}{\end{eqnarray}}
\numberwithin{equation}{section}

\begin{document}

\title{Fourier transformation and stability of differential equation on $L^1(\Bbb{R})$}

\author{ H. Rezaei$^{\ast}$ and Z. Zafarasa}

\address{Department of Mathematics, College of Sciences, Yasouj University,
Yasouj-75914-74831, Iran}

\email{rezaei@yu.ac.ir}
\email{zzafarasa@gmail.com}

\thanks{ 
\newline * Corresponding Author}

\maketitle

\begin{abstract}
In the present paper by the Fourier transform we show that every linear differential equations of $n$-th order has a solution in $L^1(\Bbb{R})$ which is infinitely differentiable in $\Bbb{R} \setminus \{0\}$.
Moreover the Hyers-Ulam stability of such equations on $L^1(\Bbb{R})$ is investigated.
\end{abstract}
\section{Introduction}
One of the main mathematical problems in the study of functional equations
is the Hyers-Ulam stability of the equations. S. M. Ulam was the first one who
investigated this problem by the question: Under what conditions does there exist
an additive mapping near an approximately additive mapping? D.H. Hyers in \cite{He}
gave an answer to the problem of Ulam for additive functions defined on Banach
spaces. Let $X_1, X_2$ be two real Banach spaces and $\ep>0$ is given. Then for every
mapping $f:X_1 \ri X_2$ satisfying
$$
\|f(x + y)-f(x)-f(y)\| < \ep \mbox{ for all } x; y \in X_1;
$$
there exists a unique additive mapping $g:X_1 \ri X_2$ with the property
$$
\|f(x)-g(x)\| < \ep \mbox{ for all } x \in X_1:
$$
After Hyers result, these problems have been extended to other functional equations \cite{Jung}.
The may be most important extension is the Hyers-Ulam stability of the differential equations.

The differential equation $\ph(y,y', \ldots,y^{(n)})=f$ has Hyers-Ulam stability on normed space $X$ if for given $\ep>0$ and a function $y$ such that
$\|\ph(y,y', \ldots,y^{(n)})-f\|<\ep$, there is a solution $y_a \in X$ of the differential equation such that $\|y-y_a\|<K(\ep)$,
and $\lim_{\ep \ri 0} K(\ep)=0$.

R. Ger \cite{AlGe} were the first authors who investigated the Hyers-Ulam stability of a differential equation.
The result of Alsina and Ger was extended to the Hyers-Ulam stability of higher order with constant coefficients in \cite{TTMM}. We mention here only the recent contributions in this subject of some mathematicians such as \cite{CiPo}, \cite{LiSh}, \cite{Ru1}, \cite{Ru2}, \cite{TTMM}, \cite{WZS}.
I.A. Rus has proved some results on the stability of ordinary differential and integral
equations using Gronwall lemma and the technique of weakly Picard operators
\cite{Ru1, Ru2}. Using the method of integral factors, the Hyers-Ulam stability of some ordinary differential equations of first and second order with constant coefficients have been proved in \cite{LiSh, WZS}.

Recall that the Laplace transform is a powerful integral transform used to switch a function from the time domain to the $s$-domain. They maps a function to a new function on the complex plane. The Laplace transform can be used in some cases to solve linear differential equations with given initial conditions.
Applying the Laplace transform method, Rezaei et al. \cite{RJR} investigated the Hyers-Ulam stability of the linear differential equations of functions defined on $(0, +\infty)$.

Fourier transforms can also be applied to the solution of differential equations of functions with domain $(-\infty, +\infty)$.
They can convert a function to a new function on the real line. Since the Laplace transform can not be used for the functions defined on $(-\infty, +\infty)$, in the present paper, we apply the Fourier transforms to show that
every $n$-order linear differential equation with constant coefficients has a solution in $L^1(\mathbb{R})$
which is infinitely differentiable in $\mathbb{R}\setminus \{0\}$.
Moreover it proves the Hyers--Ulam stability of equation on $L^1(\mathbb{R})$.
\section{Fourier Transform and Inversion Formula}
Throughout this paper $\mathbb{F}$ will denote either the real field, $\mathbb{R}$, or the complex field, $\mathbb{C}$.
Assume that $f:(-\infty, +\infty) \ri \mathbb{F}$ is absolutely integrable. Then the
Fourier transform associated to $f$ is a mapping $\mathcal{F}(f):\mathbb{R} \ri \mathbb{C}$ defined by
\be
\mathcal{F}(f)(w) = \int_{-\infty}^{\infty} e^{-iwt} f(t) dt \no
\ee
Also the inverse Fourier transform associated to a function $f \in L^1(\mathbb{R})$ is defined by
$$
\mathcal{F}^{-1}(f)(t)=\int_{-\infty}^{\infty}e^{iwt}f(w) dw
$$
By the Fourier Inversion Theorem (see Theorem A.14 in \cite{Sh}) if $f \in L^1(\mathbb{R})$ and $f$ and $f'$ are piecewise continuous on $\mathbb{R}$, that is, both are continuous in any finite interval except possibly for a finite number of jump
discontinuities, then at each point $t$ where $f$ is continuous,
$$
\mathcal{F}^{-1}(\mathcal{F}(f))(t)=f(t)
$$
In particular if $f, \mathcal{F}(f) \in L^1(\mathbb{R})$ then the above relation holds almost every where on $\mathbb{R}$
(see Theorem 9.11 in \cite{Rud}).

In the following some required properties of the Fourier transform is presented.
\begin{prop}
Let $f \in L^1(\mathbb{R})$, $F=\mathcal{F}(f)$, and $u$ be the step function defined by $u(t)=1$ for $t \geq 0$ and $u(t)=0$ for $t<0$ then
\\

{\rm (i)}  $\mathcal{F}(e^{-zt}u(t))(w)=\frac{1}{iw+z}$ provided that $Re(z)>0$,
\\

{\rm (ii)} $\mathcal{F}(f(\frac{t-t_0}{k})(w)=ke^{-iwt_0}F(kw)$,
\\

{\rm (iii)}  $\mathcal{F}(f(t)e^{iw_0t})(w)=F(w-w_0)$,
\\

{\rm (iv)} $\mathcal{F}((-it)^n f(t))(w)=F^{(n)}(w)$.
\\

{\rm (v)}  $\mathcal{F}(y^{(n)}(t))(w)=(iw)^n\mathcal{F}(y)(w)$,

\end{prop}
\begin{proof}
Part (i)-(iii) is obtained by the definition of Fourier transform. For part (iv) see Theorem 1.6, page 136, in \cite{Ka}.
For part (v) see Theorem 3.3.1, part (f), in \cite{De}.
\end{proof}
Recall that the convolution of two functions $f, g \in L^1(\mathbb{R})$ is defined by
$$
(f \ast g)(t)=\int_{-\infty}^{+\infty}f(t-x)g(x) dx.
$$
Then $f \ast g \in L^1(\mathbb{R})$, $\|f \ast g\|_1 \leq \|f\|_1 \|g\|_1$, and
$$
\mathcal{F}(f \ast g)= \mathcal{F}(f).\mathcal{F}(g).
$$
Moreover we have the following Theorem:
\begin{thm}
Let $f, g \in L^1(\mathbb{R})$ then
\\

{\rm (i)} $\mathcal{F}(f \ast g)= \mathcal{F}(f) \mathcal{F}(g)$,
\\

{\rm (ii)} $\mathcal{F}^{-1}(fg)= \mathcal{F}^{-1}(f)\ast \mathcal{F}^{-1}(g)$,
\\

{\rm (iii)} If either $f$ or $g$ is differentiable, then $f \ast g$ is differentiable.
If $f'$ exists and is a continuous, then $(f\ast g)'=f' \ast g$.
\\
\end{thm}
\begin{proof}
For part (i) see Theorems 1.3 and 1.5, page 135 in \cite{Ka}.
Part (ii) follows from the definition of $\mathcal{F}^{-1}$ and the Fourier Inversion Theorem. To see part (iii),
use Proposition 2.1, page 68, in \cite{Co}.
\end{proof}
Note that if $u$ is the step function then
$$
(f \ast u)(t)=(u \ast f)(t)=\int_{-\infty}^{+\infty}u(t-x)f(x) dx=\int_{-\infty}^{t}f(x) dx.
$$
The following corollary is deduced from Theorem 1.5, page 135, in \cite{Ka}:

\begin{cor}
Let $f \in L^1(\mathbb{R})$ then
$$
\mathcal{F}(f \ast u)(w)=\frac{\mathcal{F}(f)(w)}{iw} \quad (w \neq 0)
$$
\end{cor}

\section{Hyers-Ulam stability of the linear differential equation}
Before stating the main theorem we need the following important lemma:
\begin{prop}
Let $f \in L^1(\mathbb{R})$, $p$ be a polynomial with the complex roots $w_0, w_1, \ldots, w_{k-1}$, $k\geq 1$. Then there is a function $y_0 \in L^1(\mathbb{R})$ which is
infinitely differentiable in $\mathbb{R} \setminus \{0\}$, $(k-1)$-times differentiable at zero, and satisfying
\be
\mathcal{F}(y_0)(w)=\frac{\mathcal{F}(f)(w)}{p(w)} \quad (w \neq w_0, w_1, \ldots, w_{k-1}).
\ee
\end{prop}
\begin{proof}
First assume that $p(w)=(w-w_0)^k$ and $Im(w_0) \neq 0$. Put $z_1:=-iw_0$ when $Im(w_0)>0$ and $z_2:=iw_0$ when $Im(w_0)<0$.
Since $Re(z_i)>0, i=1,2$, Proposition 2.1, part (i), implies that
$$
\mathcal{F}(e^{-z_1t}u(t))(w)=\frac{1}{iw+z_1}=\frac{1}{iw-iw_0}=\frac{1}{i(w-w_0)},
$$
and
$$
\mathcal{F}(e^{-z_2t}u(t))(w)=\frac{1}{iw+z_2}=\frac{1}{iw+iw_0}=\frac{1}{i(w+w_0)}.
$$
In the second case using part (ii) of Proposition 2.1 for $t_0=0$ and $k=-1$ to get
$$
\mathcal{F}(e^{z_2t}u(-t))(w)=\frac{1}{-iw+iw_0}=\frac{-1}{i(w-w_0)}.
$$
According the above computations the function $f_0:(-\infty, +\infty) \ri \mathbb{F}$ defined by

\[
f_0(t)= \left\{\begin{array}{ll}

ie^{iw_0t}u(t) & \mbox{if} \quad Im(w_0)>0\\

-ie^{iw_0t}u(-t)& \mbox{if} \quad Im(w_0)<0 \\
\end{array}\right.
\]
satisfies in the equation
$$
\mathcal{F}(f_0)(w)=\frac{1}{w-w_0}
$$
and by part (iv) of Proposition 2.1, we see that
$$
\mathcal{F}\Big((-it)^{k-1} f_0 \Big)(w)=\Big[\frac{1}{(w-w_0)}\Big]^{(k-1)}=\frac{(-1)^{k-1} (k-1)!}{(w-w_0)^{k}}.
$$
Now put
$$
y_{0}:=\Big(\frac{(-it)^{k-1}}{(k-1)!}f_0\Big) \ast f
$$
then $y_0 \in L^1(\mathbb{R})$, by Theorem 2.2 part (ii) $y_0$ is infinitely differentiable in $\mathbb{R} \setminus \{0\}$, $y^{(k-1)}_0(0)$ exists and
$$
\mathcal{F}(y_{0})(w)=\mathcal{F}\Big( \frac{(-it)^{k-1}}{(k-1)!}f \Big)(w) \mathcal{F}(y)(w)
=\frac{\mathcal{F}(y)(w)}{(w-w_0)^m}
$$
for $w \neq w_0$.
In the next case suppose that $Im(w_0)=0$, namely, $w_0$ is a real number.
In this case the argument is different. Indeed, let $f_0=f$ and
$$
f_{i+1}(t)=e^{iw_0 t} \big[(f_{i} e^{-iw_0 t}) \ast u(t) \big]
$$
for $i=0,1,2, \ldots k$. Then by Corollary 2.3, we have
$$
\mathcal{F}(f \ast u)(w)= \frac{\mathcal{F}(f)(w)}{iw} \quad( w \neq 0)
$$
and so by (iii) of Proposition 2.1:
$$
\mathcal{F}\Big(f(t)e^{-iw_0 t} \ast u \Big)(w)=\frac{\mathcal{F}(f(t)e^{-iw_0 t})(w)}{iw}=\frac{\mathcal{F}(y)(w+w_0)}{iw}.
$$
Hence
$$
\mathcal{F}(f_1)(w)=\mathcal{F}\Big(e^{iw_0 t} \big[(f(t) e^{-iw_0 t}) \ast u(t) \big]\Big)(w)
=\frac{\mathcal{F}(f)(w)}{i(w-w_0)}.
$$
Continuous in this way to get
$$
\mathcal{F}(f_{i+1})(w)=\frac{\mathcal{F}(f_i)(w)}{i(w-w_0)}.
$$
Therefore
$$
\mathcal{F}(f_{k})(w)=\frac{\mathcal{F}(w)(y)}{i^{k}(w-w_0)^k}.
$$
Finally, put $y_0:=i^{k} f_k$ then $y_0$ has the requested conditions and satisfies in (3.1) for $p(w)=(w-w_0)^k$ and in this case the proof is completed.

Now suppose that $p(w)$ expresses as a product of linear factors; $$p(w)=(w-w_1)^{k_1}(w-w_2)^{k_2} \ldots (w-w_k)^{n_k}$$
for some complex number $w_i, i=1, \ldots k$ and some integer $k_i, i=1, 2, \ldots, k$.
Applying the partial fraction decomposition of $\frac{1}{P(w)}$, we obtain:
$$\frac{1}{p(w)}=\sum_{i=1}^{k}\sum_{j=1}^{n_i}\frac{\la_{ij}}{(w-w_i)^{j}}$$
where $\la_{ij}$ is complex number for $i=1,2, \ldots ,k$ and $j=1,2, \ldots,n_i$.
Considering the first part of proof there exists a $y_{ij} \in L^1(\mathbb{R})$ which is
$k$-times differentiable on $\mathbb{R} \setminus \{0\}$ and
\be
\mathcal{F}(y_{ij})(w)=\frac{\mathcal{F}(y)}{(w-w_i)^j} \quad (w \neq w_i).
\ee
for every integer $1 \leq i \leq k$ and $1 \leq j \leq n_i$. Then
put
\be
y_0(w)=\sum_{i=1}^{k}\sum_{j=1}^{n_i} \la_{ij}y_{ij}(w) \quad (w \neq w_1, w_2, \ldots, w_k)\no
\ee
and use the linearity of Fourier transform and (3.2) to see
\begin{align}
\mathcal{F}(y_0)(w)=\mathcal{F} \bigg (\sum_{i=1}^{k}\sum_{j=1}^{n_i} \la_{ij}y_{ij}(w) \bigg)
= \sum_{i=1}^{k}\sum_{j=1}^{n_i} \la_{ij} \mathcal{F}(y_{ij})(w) \no \\
= \sum_{i=1}^{k}\sum_{j=1}^{n_i}\frac{\la_{ij} \mathcal{F}(y)}{(w-w_i)^{j}}=\frac{\mathcal{F}(y)(w)}{p(w)}. \no
\end{align}
\end{proof}

\begin{thm}
Consider the differential equation
\be
y'(t)+a_0 y(t)=f(t) \quad (t \in \mathbb{R}) \no
\ee
where $a_0$ is a nonzero scalar such that $Re(a_0) \neq 0$ and $f \in L^1(\mathbb{R})$.
Then there exists a constant $M$ with the following property:
For every $y \in L^1(\mathbb{R})$ and for given $\ep>0$ satisfying
\be
\|y'+a_{0} y -f\|_1 \leq \ep;\no
\ee
there exists a differentiable solution $y_a \in L^1(\mathbb{R})$ of equation such that $\big \|y_a-y \big \|_1 \leq M \ep$.
\end{thm}
\begin{proof}
Let
\be
h(t)=y'(t)+a_{0} y(t) -f(t). \no
\ee
Then
\be
\mathcal{F}(h)(w)&=&\mathcal{F}(y')(w)+a_0\mathcal{F}(y)(w)-\mathcal{F}(f)(w) \no \\
&=&(iw)\mathcal{F}(y)(w)+a_0\mathcal{F}(y)(w)-\mathcal{F}(f)(w) \no \\
&=&(iw+a_0)\mathcal{F}(y)(w)-\mathcal{F}(f)(w) \no
\ee
Hence
\be
\mathcal{F}(y)(w)-\frac{\mathcal{F}(f)(w)}{iw+a_0}=\frac{\mathcal{F}(h)(w)}{iw+a_0}.
\ee
By the preceding proposition there exists a function $y_a \in L^1(\mathbb{R})$ such that
$$
\mathcal{F}(y_a)(w)=\frac{\mathcal{F}(f)(w)}{iw+a_0}.
$$
Without loss of generality suppose that $Re(a_0)>0$ then by considering (3.3) and part (i) of Proposition 2.1, we obtain
\be
\mathcal{F}(y-y_a)(w)&=&\mathcal{F}(y)(w)-\mathcal{F}(y_a)(w) \no \\
&=&\mathcal{F}(y)(w)-\frac{\mathcal{F}(f)(w)}{iw+a_0} \no \\
&=&\frac{\mathcal{F}(h(t))(w)}{iw+a_0}=\mathcal{F}(e^{-a_{0}t} u(t) \ast h(t))(w). \no
\ee
Consequently $y(t)-y_a(t)=e^{-a_{0}t} u(t) \ast h(t)$ and
\be
\|y-y_a\|_1&=&\int_{-\infty}^{+\infty}|y(t)-y_a(t)| dt \no \\
&=& \int_{-\infty}^{+\infty} \Big |e^{-a_{0}t} u(t) \ast h(t)\Big | dt\no \\
&\leq& \int_{-\infty}^{+\infty} \Big |\int_{-\infty}^{+\infty} e^{-x a_0} u(x) h(t-x) dx \Big | dt \no \\
&\leq& \int_{-\infty}^{+\infty} \int_{-\infty}^{+\infty} |e^{-x a_0}|u(x)||h(t-x)|dx dt \no \\
&\leq& \int_{\infty}^{+\infty} \int_{0}^{+\infty} e^{-x Re(a_0)}|h(t-x)|dx dt\no \\
&\leq& \int_{0}^{+\infty} \int_{-\infty}^{+\infty} e^{-x Re(a_0)}|h(t-x)|dt dx \leq \frac{1}{Re(a_0)}\|h\|_1<\frac{\ep}{Re(a_0)}\no
\ee
This completes the proof.
\end{proof}

We denote by $C^n(\mathbb{R})$ the space of all $n$-times differentiable continuous functions on $\mathbb{R}$.
\begin{thm}
Consider the differential equation
\be
y^{(n)}(t)+\sum_{k=0}^{n-1}\al_{k} y^{(k)}(t)=f(t)
\ee
where $f \in L^1(\mathbb{R})$, $n>1$, $\al_0, \al_1, \ldots \al_{n-1}$ are given scalars such that $Re(a_{n-1})\neq 0$.
Then there exists a constant $M$ with the following property:
For every $y \in L^1(\mathbb{R})$ and for given $\ep>0$ satisfying
\be
\bigg \| y^{(n)}+\sum_{k=0}^{n-1}\al_{k} y^{(k)} -f \bigg \|_1 \leq \ep;
\ee
there exists a solution $y_a$ of the equation {\rm (3.4)} such that $y_a \in C^n(\mathbb{R}) \cap L^1(\mathbb{R})$ and
$$
\big \|y_a-y \big \|_1\leq M \ep \quad (t \in \mathbb{R}),
$$
\end{thm}
\begin{proof}
Let
\be
h(t)=y^{(n)}(t)+\sum_{k=0}^{n-1}\al_{k} y^{(k)}(t)-f(t).
\ee
Applying Proposition 2.1, part (v), we may derive
\be
\mathcal{F}(h)(w)&=&\mathcal{F}(y^{(n)})(w)+\sum_{k=0}^{n-1}\al_{k} \mathcal{F}(y^{(k)})(w)-\mathcal{F}(f)(w)  \\
&=&(iw)^{n}\mathcal{F}(y)+\sum_{k=0}^{n-1}\al_{k}(iw)^{k}\mathcal{F}(y)-\mathcal{F}(f)(w) \no \\
&=&p(iw)\mathcal{F}(y)(w)-\mathcal{F}(f)(w)\no
\ee
where $p$ is a complex polynomial determined by
\be
p(z)=z^n+\sum_{k=0}^{n-1}\al_{k}z^k
\ee
Applying the relations (3.6) and (3.7) for $h=0$ to find that
$y$ is a solution of (3.4) if and only if
\be
\mathcal{F}(y)(w)=\frac{\mathcal{F}(f)(w)}{p(iw)}.
\ee
Now from (3.7) we deduce that
\be
\frac{\mathcal{F}(h)(w)}{p(iw)}=\mathcal{F}(y)(w)-\frac{\mathcal{F}(f)(w)}{p(iw)} \quad (iw \neq w_0, w_1, \ldots, w_{n-1})
\ee
where $w_0, w_1, \ldots, w_{n-1}$ are the roots of the polynomial $p(w)$.
By Proposition 3.1, there exists a function $y_a \in L^1(\mathbb{R})$ such that
$$
\mathcal{F}(y_a)(w)=\frac{\mathcal{F}(f)(w)}{p(iw)}.
$$
Since every function satisfying (3.9) is a solution of (3.4), we find that $y_a$ is a solution of equation (3.4)
and from (3.10) we obtain
$$
\mathcal{F}(y)(w)-\mathcal{F}(y_a)(w)=\frac{\mathcal{F}(h)(w)}{p(iw)},
$$
and consequently
\be
|y(t)-y_a(t)|=\Big |\mathcal{F}^{-1}(\frac{\mathcal{F}(h)(w)}{p(iw)})(t)\Big |. \no
\ee
By the definition of $h$ and the inequality (3.5), $\|h\|_1\leq \ep$, so
\be
|\mathcal{F}(h)(w)| \leq \int_{-\infty}^{+\infty} |e^{-iwt}||h(t)| dt \leq \ep.
\ee
Now by considering the part (i) of Theorem 2.2, we have
\be
\|y-y_a\|_1&=&\int_{-\infty}^{+\infty}|y(t)-y_a(t)| dt \no \\
&=&\int_{-\infty}^{+\infty} \bigg |\mathcal{F}^{-1}(\frac{\mathcal{F}(h)(w)}{p(iw)})\bigg |dt \no \\
&=&\int_{-\infty}^{+\infty} \bigg |\mathcal{F}^{-1}(\mathcal{F}(h)(w))(t)
\ast \mathcal{F}^{-1}( \frac{1}{p(iw)})(t)\bigg |dt \no \\
&=&\int_{-\infty}^{+\infty} \bigg |(h \ast \mathcal{F}^{-1}(\frac{1}{p(iw)}))(t)\bigg |dt \no \\
&=& \Big \|h \ast \mathcal{F}^{-1}(\frac{1}{p(iw)})\Big \|_1 \no \\
&\leq& \big \|h \big \|_1 \Big\|\mathcal{F}^{-1}(\frac{1}{p(iw)}) \Big\|_1=M \ep \no
\ee
where
$$
M=\Big \|\mathcal{F}^{-1}(\frac{1}{p(iw)}) \Big\|_1<+\infty \quad (\mbox{ because } deg(p)>1).
$$
To see the last relation, looking (3.8) and the fact that $w_0, w_1, \ldots, w_{n-1}$ are the roots of polynomial $p$, we get
$$
w_0+w_1+\ldots w_{n-1}=-a_{n-1}
$$
and
$$
p(iw)=(iw-w_0)(iw-w_1)\ldots(iw-w_{n-1}).
$$
Since $Re(a_{n-1}) \neq 0$ there exists some $0 \leq j \leq n-1$ such that $Re(w_j) \neq 0$.
Then by Proposition 2.1, part (i), there exists some $y_j \in L^{1}(\mathbb{R})$ such that
$$
\mathcal{F}(y_j)(w)=\frac{1}{iw-w_j}
$$
Let $q(w)=\prod_{k \neq j}(iw-w_k)$. Then
$$
\frac{1}{p(iw)}=\frac{\frac{1}{iw-w_j}}{q(w)}=\frac{\mathcal{F}(y_j)(w)}{q(w)}.
$$
Since $y_j \in L^1(\mathbb{R})$ by Proposition 3.1 there exists some $y_j \in L^1(\mathbb{R})$ such that
$$
\mathcal{F}(y_j)(w)=\frac{\mathcal{F}(y_j)(w)}{q(w)}
$$
Comparing the above relations, we see that
$$
\mathcal{F}(y_j)(w)=\frac{1}{p(iw)}
$$
and hence
$$
\frac{1}{p(iw)}=\mathcal{F}^{-1}(y_j(w)) \in L^1(\mathbb{R}).
$$

Therefore (3.4) has Hyers-Ulam stablility and the proof is completed.
\end{proof}
The following example shows that the condition $Re(a_{n-1})\neq 0$ in the above theorem is necessary and can not be removed:
\begin{exa}
{\rm
Every nonzero solution of equation $y'(t)-iy(t)=0$ has the form $y=\la e^{it}$ which is not in $L^1(\mathbb{R})$.
Hence $y=0$ is the only solution of equation in $L^1(\mathbb{R})$. Now for $\ep>0$ consider the function
$$
y_{\ep}(t)=e^{(i-1)t}u(t)+\frac{\ep}{\sqrt 2} e^{-t}u(t)
$$
which is in $L^1(\mathbb{R})$. To see this note that
\be
\int_{-\infty}^{+\infty} |y_{\ep}(t)| dt&=&\int_{-\infty}^{+\infty} |e^{(i-1)t}u(t)+\frac{\ep}{\sqrt 2} e^{-t}u(t)|dt \no \\
&\leq&\int_{-\infty}^{+\infty} |e^{(i-1)t}u(t)| dt+\frac{\ep}{\sqrt 2}\int_{-\infty}^{+\infty}|e^{-t}u(t)|dt \no \\
&=& \int_{0}^{+\infty} |e^{it}||e^{-t}| dt+\frac{\ep}{\sqrt 2}\int_{0}^{+\infty}|e^{-t}|dt \no \\
&=& \int_{0}^{+\infty} e^{-t} dt+\frac{\ep}{\sqrt 2}\int_{0}^{+\infty}e^{-t}dt<+\infty \no
\ee
On the other hand, for $t < 0$, $y'_{\ep}(t)=y_{\ep}(t)=0$.
Thus
\be
\|y'_{\ep}-iy_{\ep}\|_1&=&\int_{-\infty}^{+\infty}|y'_{\ep}(t)-iy_{\ep}(t)| dt \no \\
&=&\int_{-\infty}^{+\infty}|ie^{it}-\frac{\ep}{\sqrt 2} e^{-t}u(t)-iMe^{it}-\frac{\ep}{\sqrt 2}ie^{-t}u(t)| \no \\
&=&\frac{\ep}{\sqrt 2} \int_{0}^{+\infty}|1+i|e^{-t} dt=\ep
\ee
If the equation $y'(t)-iy(t)=0$ has Hyers-Ulam stability, it must be a solution $y_a \in L^1(\mathbb{R})$ of the differential equation such that $\|y_{\ep}-y_a\|_1<K(\ep)$ and $\lim_{\ep \ri 0} K(\ep)=0$.
Since the only solution of equation in $L^1(\mathbb{R})$ is zero function, so $y_a=0$ and on the other hand
\be
\|y_{\ep}-y_a\|_1=\|y_{\ep}\|_1&=&\int_{0}^{+\infty}|e^{(i-1)t}+\frac{\ep}{\sqrt 2} e^{-t}| \no \\
&\geq& \int_{0}^{+\infty}|e^{(i-1)t}|dt-\frac{\ep}{\sqrt 2}\int_{0}^{+\infty}e^{-t} dt=1-\frac{\ep}{\sqrt 2}. \no
\ee
Therefore
$$
1-\frac{\ep}{\sqrt 2} \leq K(\ep)
$$
which is a contradiction with the fact that $\lim_{\ep \ri 0} K(\ep)=0$.
}
\end{exa}

\end{document}